\documentclass[draft,12pt]{amsart}
\usepackage{amscd}

\usepackage{latexsym}

\usepackage{amsfonts}
\usepackage{amssymb}

\newcommand{\norm}[1]{\lVert#1\rVert}
\newcommand{\bignorm}[1]{\left\lVert#1\right\rVert}

\theoremstyle{plain}
\newtheorem{theorem}{Theorem}[section]
\newtheorem{corollary}[theorem]{Corollary}
\newtheorem{lemma}[theorem]{Lemma}
\newtheorem{proposition}[theorem]{Proposition}

\theoremstyle{plain}

\newtheorem{remark}[theorem]{Remark}

\newtheorem{example}[theorem]{Example}

\def\R{{\rm I\kern-2ptR}}
\def\N{{\rm I\kern-2ptN}}

\def\emptyset{\mathop{\raisebox{.1ex}{$\not
\mathrel{\raisebox{.1ex}{$\scriptstyle\bigcirc$}}$}}}

\def\qed{\hskip .6em \raise1.8pt\hbox{\vrule height4pt
width6pt depth2pt}}
\def\qedd{\hskip .4em \raise1.8pt\hbox{\vrule height3pt
width5pt depth1.8pt}}
\def\sq{\hskip .6em \raise1.8pt\hbox{\vrule
height4pt width6pt depth2pt}}

\def\leaderfill{\leaders\hbox to 1em{\hss.\hss}\hfill}

\begin{document}

\title[]{Almost-invariant and essentially-invariant halfspaces}

\author{Gleb Sirotkin}
\address{Department of Mathematical Sciences, Northern Illinois University, DeKalb, IL 60115}
\email{gsirotkin@niu.edu}

\author{Ben Wallis}
\address{Department of Mathematical Sciences, Northern Illinois University, DeKalb, IL 60115}
\email{bwallis@niu.edu}

\begin{abstract} 
In this paper we study sufficient conditions for an operator to have an almost-invariant half-space.
As a consequence, we show that if $X$ is an infinite-dimensional complex Banach space then every operator $T\in\mathcal{L}(X)$ admits an essentially-invariant half-space. We also show that whenever a closed algebra
of operators possesses a common AIHS, then it has a common invariant half-space as well. 
\end{abstract}
 
\thanks{The authors thank William B. Johnson for his helpful comments.}

\maketitle
\theoremstyle{plain}
%%%%%%%%%%%%%%%%%%%%%%%%%%%%%%%%%%%%%%%%%%%%%%Introduction

\section{Introduction and the first result}

The invariant subspace problem in its full generality was famously closed in the negative by Enflo in a series of papers beginning with \cite{En76} and ending with \cite{En87}, where he constructed an operator acting on a separable Banach space which fails to admit any nontrivial invariant subspace, that is, invariant subspace which is nonzero, proper, and closed.  
Independently Charles Read produced an outstanding series of examples \cite{Re84, Re85, Re86, Re91, Re97}, in particular, an operator acting on $\ell_1$  \cite{Re85} which fails to admit any nontrivial invariant subspace. Thus, the answer is still negative even for ``nice'' Banach spaces.  Current work on the invariant subspace problem now focuses on the separable Hilbert space case, which remains unsolved for the time being.

A related but independent problem was posed in \cite{APTT09}.  Let us say that a subspace $Y$ of a Banach space $X$ is {\bf almost-invariant} under $T$ whenever $TY\subseteq Y+E$ for some finite-dimensional {\bf error} subspace $E$.  In this case, the smallest possible dimension of $E$ we call the {\bf defect}.  To make things nontrivial, we restrict our attention to the case where $Y$ is a {\bf halfspace}, that is, a closed subspace with both infinite dimension and infinite codimension in $X$.  Let us use the abbreviations {\bf AIHS} for ``almost-invariant halfspace'' and {\bf IHS} for ``invariant halfspace.''  Then we can ask, does every operator on an infinite-dimensional Banach space admit an AIHS?  Let us call this the {\bf AIHS problem}.

It turns out that for closed algebras of operators existence of a common AIHS and IHS is equivalent. In \cite[Theorem 2.3]{MPR13}, this was proved for the case when AIHS is complemented, but the proof can be adapted to the general case as well. Let us demonstrate this.  Note also that, although our other results are validated only for complex Banach spaces, this next theorem works even for real Banach spaces.

\begin{theorem}\label{common-IHS}Let $X$ be a (real or complex) Banach space, and let $\mathcal{A}$ be a norm-closed algebra of operators in $\mathcal{L}(X)$.  If there exists a halfspace $Y$ of $X$ which is almost-invariant under every $A\in\mathcal{A}$, then there exists a halfspace $Z$ which is invariant under every $A\in\mathcal{A}$.\end{theorem}

\begin{proof}Let us assume, without loss of generality, that $\mathcal{A}$ is a unital algebra satisfying the hypothesis of the theorem. By a result of Popov (\cite[Theorem 2.7]{Po10}), there is number $M\in\mathbb{N}$ such that for every $A\in\mathcal{A}$ we have $AY\subseteq Y+F_A$ with $\text{dim}(F_A)\leq M$.  

For each $A\in\mathcal{A}$, define $J_A\in\mathcal{L}( Y, X/ Y)$ by $J_A:=qA|_{ Y}$, where $q: X\to X/ Y$ is the canonical quotient map. Let us also assume that there are vectors $y_1, \dots, y_M \in Y$ and that there is $T\in \mathcal{A}$ such that for $W := J_T( [y_i]_{i=1}^M)$ we have $\text{dim}(W)= M$. It was proved in \cite[Lemma 3.4]{Po10} that if $Y$ is an AIHS under $T$ and $V:=Y\cap T^{-1}(Y)$ then the defect of $Y$ under $T$ is precisely $\text{dim}(Y/V)$.  Thus, $\text{dim}(Y/V)=M$.  Notice that $V=\mathcal{N}(J_T)$ and $Y = V \oplus [y_i]_{i=1}^M$. 

Fix any $A \in \mathcal{A}$ and any vector $v \in V$ and let us show that $J_Av \in W$ holds. Indeed, if not, let $\delta>0$ be the minimum of 
$\|\sum_{i = 1}^M \alpha_i J_Ty_i + \alpha_{M+1}J_Av\|$ over the set of all $(\alpha_i)_{i=1}^{M+1}$ from the unit sphere of $\mathbb{K}^{M+1}$. The minimum is attained
due to the compactness of the sphere and, thus, strictly positive due to the linear independence of the vectors $J_Ty_1,\dots, J_Ty_M, J_Av$.  
Again by compactness, the maximum of $\norm{\sum_{i=1}^M\alpha_i J_Ay_i}$ over the unit sphere in $\mathbb{K}^{M+1}$ exists.  Hence, we could find $\lambda >0$ small enough to have $\lambda\|\sum_{i = 1}^M \alpha_iJ_Ay_i \| <\delta$. Then, 
using the fact that $J_Tv=0$, 
we obtain $M+1$ linearly independent vectors
\[(\lambda J_A +  J_T)y_1,\cdots,(\lambda J_A + J_T)y_M,(\lambda J_A + J_T)\left({1 \over \lambda}v\right).\]
We can see this due to the fact that if $(\alpha_i)_{i=1}^{M+1}$ is in the unit sphere of $\mathbb{K}^{M+1}$ then we would have
\begin{multline*}\bignorm{\sum_{i=1}^M\alpha_i(\lambda J_A +  J_T)y_i+\alpha_{M+1}(\lambda J_A + J_T)\left({1 \over \lambda}v\right)}\\\geq\bignorm{\sum_{i=1}^M\alpha_iJ_Ty_i+\alpha_{M+1}J_Av}-\lambda\bignorm{\sum_{i=1}^M\alpha_iJ_Ay_i}>0.\end{multline*}
Since $\lambda J_A +  J_T = J_{\lambda A +  T}$ we obtain a contradiction to the fact that the defect of $\lambda A + T \in \mathcal{A}$ must be no larger than $M$. 

Thus, we have shown that there is an $M$-dimensional subspace $F_T\subset X$ such that for every $A \in \mathcal{A}$ we have $AV \subseteq Y+F_T$. 
Now we have the following inclusion 
\[Z=\overline{\text{span}\{Av:A\in\mathcal{A},v\in V\}}\subset Y+F_T.\]
Moreover, since $V$ is finite-codimensional in a halfspace $Y$, it is a halfspace itself. Therefore, $Z$ is infinite-dimensional as $\mathcal{A}$ is unital. 
We conclude by observing that $Z$ is a half-space which is $\mathcal{A}$-invariant due to  $\mathcal{A}$ being an algebra.\end{proof}

The connections between the existence of an invariant subspace and the existence of an AIHS for an individual operator are different.  If $T\in\mathcal{L}(X)$, $X$ an infinite-dimensional complex Banach space, fails to admit an invariant subspace, then $T$ has no eigenvalues and hence, by \cite[Theorem 2.3]{PT13}, admits an AIHS of defect $\leq 1$.  Conversely, if $T$ admits no AIHS of defect $\leq 1$ then again by \cite[Theorem 2.3]{PT13} it admits nontrivial invariant subspaces (although, not necessarily halfspaces).  

Unlike the invariant subspace problem, the AIHS problem remains open in its full generality, although significant partial results have been obtained in recent years.  Namely, in \cite{PT13}, Popov and Tcaciuc proved that every operator acting on a complex, infinite-dimensional reflexive space admits an AIHS of defect $\leq 1$, and then in \cite{SW14} the authors showed that if $X$ is an infinite-dimensional complex Banach space and $T\in\mathcal{L}(X)$ is quasinilpotent or weakly compact then it admits an AIHS of defect $\leq 1$.  Together with \cite[Remark 2.9]{MPR12} and \cite[Theorem 2.3]{PT13}, it follows that if $T$ has only countably many eigenvalues (for example, if $T$ is strictly singular) then it admits an AIHS of defect $\leq 1$.  In section 2, we will show that if $T$ commutes with a strictly singular operator with infinite-dimensional range then it admits an AIHS of defect $\leq 1$.  Note that results of this kind draw heavily upon the tools of spectral theory, and so do not obviously carry over to the case where $X$ is a real Banach space.

Suppose the answer to the AIHS problem is negative in the complex case.  In section 3, we describe the structure of a hypothetical operator acting on an infinite-dimensional complex Banach space, but which fails to admit an AIHS of defect $\leq 1$.  In particular, it simultaneously exhibits left-shift-like and right-shift-like behavior, in a certain sense which we will make explicit later on.  This may help to guide future construction of a counter-example, that is, an operator acting on an infinite-dimensional complex Banach space which fails to admit an AIHS.

In lieu of a solution to the AIHS problem, we can consider a weaker condition to almost-invariance.  In \cite[Proposition 1.3]{APTT09} the authors observed that $Y$ is almost-invariant under $T\in\mathcal{L}(X)$ with defect $\leq d$ if and only if $Y$ is invariant under $T+F$ for some $F\in\mathcal{F}(X)$ of rank $\leq d$, where $\mathcal{F}$ denotes the class of finite-rank operators.  Let us say that a $Y$ is {\bf essentially-invariant} under $T\in\mathcal{L}(X)$ just in case it is invariant under $T+K$ for some $K\in\mathcal{K}(X)$, where $\mathcal{K}$ denotes the class of compact operators.  We can then ask, does every operator acting on an infinite-dimensional Banach space admit an essentially-invariant halfspace ({\bf EIHS})?  Note that every AIHS is an EIHS, and making this a formally weaker question.

To the best of our knowledge, EIHS's made their first appearance in the literature in 1971, in a paper by Brown and Pearcy.  There, they showed that every operator on a complex infinite-dimensional Hilbert space admits an EIHS (\cite[Corollary 3.2]{BP71}).  Later, in \cite{Ho78}, Hoffman summarized the results of a few earlier papers, namely \cite{FSW72}, \cite{Vo76}, and \cite{Ar77}.  By that time he was able to observe that every operator on a separable, infinite-dimensional complex Hilbert space is essentially reducing, that is, admits a halfspace $Y$ such that both $Y$ and its orthogonal complement are essentially-invariant under the operator.
In section 4, we will show further that every operator acting on an infinite-dimensional complex Banach space admits an EIHS.  Using the same techniques, we will also prove that for any infinite-dimensional complex Banach space $X$, the set of operators admitting an AIHS of defect $\leq 1$ is norm-dense in $\mathcal{L}(X)$.

All notation in this paper is standard, such as seen in \cite{LT77}, \cite{AA02}, and \cite{AK06}.
Let us just mention three kinds of spectrum of an operator $T:X\to X$ that we constantly use in this paper: 
the spectrum $\sigma(T) = \{\lambda \in \mathbb{C}: \lambda - T \text{ is not invertible}\}$, 
the point spectrum $\sigma_p(T) = \{\lambda \in \mathbb{C}: \lambda - T \text{ is not injective}\}$, and
the surjectivity spectrum $\sigma_{su}(T) = \{\lambda \in \mathbb{C}: \lambda - T \text{ is not surjective}\}$.

\section{An analog for Lomonosov's Theorem}

Let us formally state and prove an important result which was alluded to but not explicitly given in \cite{SW14}.

\begin{theorem}[\cite{SW14}]\label{countable}Let $X$ be an infinite-dimensional complex Banach space, and let $T\in\mathcal{L}(X)$.  If $T$ has no more than countably many eigenvalues then it admits an AIHS of defect $\leq 1$.\end{theorem}

\begin{proof}In \cite[Remark 2.9]{MPR12} it was observed that if $T$ fails to admit an IHS then there exists a finite-codimensional $T$-invariant subspace $W$ such that $\sigma(T|_W)$ is a connected component of $\sigma(T)$.\footnote{This Remark appears in the preprint \cite{MPR12}, but not the published version \cite{MPR13}.  Its removal has to do with organizational purposes and the near-simultaneous publication of \cite{PT13} rather than any logical problem.  The same thing happened for \cite[Lemma 2.1]{MPR12}, which we use in the proof of Lemma \ref{countable-commutes}, but which was also stricken from \cite{MPR13}.}  Then, it was shown in \cite[Theorem 3.11]{SW14} that every quasinilpotent operator acting on an infinite-dimensional complex Banach space admits an AIHS of defect $\leq 1$.  Hence, if $\sigma(T|_W)$ is a singleton then we are done.  Otherwise, since $\sigma(T|_W)$ is a connected and compact subset of $\mathbb{C}$ which is not a singleton, $\partial\sigma(T|_W)$ must be uncountable.  It was shown in \cite[Theorem 2.3]{PT13} that if the boundary of the spectrum contains a non-eigenvalue then the operator admits an AIHS of defect $\leq 1$.  Since $\sigma_p(T)$ and hence $\sigma_p(T|_W)$ is countable, such a non-eigenvalue must lie in the uncountable set $\partial\sigma(T|_W)$.  This gives us an AIHS of defect $\leq 1$ for $T|_W$, and hence also for $T$.\end{proof}

So, for example, any strictly singular operator acting on a complex infinite-dimensional Banach space admits an AIHS of defect $\leq 1$.  This is due to the fact that such operators always have countable spectrum (cf., e.g., \cite[Theorem 7.11]{AA02}).  However, we can say something considerably stronger.

\begin{theorem}\label{SScommutes}Let $X$ be an infinite-dimensional complex Banach space and $T\in\mathcal{L}(X)$.  If $T$ commutes with a strictly singular (or compact) operator with infinite-dimensional range, then $T$ admits an AIHS of defect $\leq 1$.\end{theorem}

\noindent This can be viewed as an analog to Lomonosov's Theorem, which states that any operator acting on an infinite-dimensional Banach space and commuting with a nonzero compact operator admits a nontrivial hyper-invariant closed subspace (\cite{Lo73}). Note that unless the AIHS problem has an affirmative answer in the complex case, simply being nonzero is insufficient.  To see this, let $T\in\mathcal{L}(X)$ be an operator acting on a complex Banach space $X$ which fails to admit an AIHS of defect $\leq 1$.  Then $T\oplus 0\in\mathcal{L}(X\oplus\mathbb{C})$ also fails to admit an AIHS of defect $\leq 1$, even though it commutes with the nonzero compact operator $0\oplus I_{\mathbb{C}}$, where $I_\mathbb{C}$ is the identity acting on $\mathbb{C}$.  On the other hand, Lomonosov's Theorem cannot be strengthened to use strictly singular operators instead of compact ones, since Read has famously constructed a strictly singular operator acting on a Banach space with no nontrivial invariant closed subspace (\cite{Re91}).

First, let us prove the following lemma.
\begin{lemma}\label{countable-commutes}Let $X$ be an infinite-dimensional complex Banach space, and suppose $S\in\mathcal{L}(X)$ and $T\in\mathcal{L}(X)$ satisfy the following properties.
\begin{itemize}\item[(i)]  $\sigma_p(S)$ is countable;
\item[(ii)]  $S$ is not a multiple of the identity, that is, $S\neq\mu I$ for any $\mu\in\mathbb{C}$;
\item[(iii)]  $\sigma_p(T^*)=\emptyset$; and
\item[(iv)]  $S$ commutes with $T$, that is, $ST=TS$.\end{itemize}
Then $T$ admits an AIHS of defect $\leq 1$.\end{lemma}

\begin{proof}We may assume $\sigma_p(T)$ is uncountable by Theorem \ref{countable}, say $\sigma_p(T)=(\lambda_\alpha)_{\alpha\in A}$ for some uncountable index set $A$.  Notice that each eigenspace $\mathcal{N}(\lambda_\alpha-T)$, $\alpha\in A$, is $S$-invariant, since given $u\in\mathcal{N}(\lambda_\alpha-T)$ we get
\[TSu=STu=\lambda_\alpha Su\]
so that $Su$ is either zero or a $\lambda_\alpha$-eigenvector under $T$.  Furthermore, we can assume that each such eigenspace is finite-dimensional since otherwise, due to the fact that every infinite-dimensional closed subspace contains a halfspace (\cite[Lemma 2.1]{MPR12}), it would contain an IHS.  Recall that any linear operator acting on a finite-dimensional complex vector space admits an eigenvalue, and hence an eigenvector.  In particular, for each $\alpha\in A$ we can find an $S$-eigenvector $v_\alpha\in\mathcal{N}(\lambda_\alpha-T)$.  This gives us an uncountable set $(v_\alpha)_{\alpha\in A}$ of linearly independent $S$-eigenvectors, which are also $T$-eigenvectors.  Let $(\mu_\alpha)_{\alpha\in A}$ be the corresponding set of 
eigenvalues under $S$.  Since $\sigma_p(S)$ is countable, it must be that $(\mu_\alpha)_{\alpha\in A}$ has only countably many distinct values.  Then we can find a sequence $(\alpha_n)$ such that $(v_{\alpha_n})$ all share the same $S$-eigenvalue, say $\mu$.  So, it cannot be that $[v_{\alpha_n}]=X$, else we would have $S=\mu I$, contradicting the fact that $S$ is not a multiple of the identity.  However, given that each $v_{\alpha_n}$ is also a $T$-eigenvector, $[v_{\alpha_n}]$ is an infinite-dimensional $T$-invariant closed subspace.  Recall that if $\sigma_p(T^*)=\emptyset$ then every infinite-dimensional $T$-invariant subspace is either the whole space or a halfspace (\cite[Proposition 3.9]{SW14}).  In particular, $[v_{\alpha_n}]$ is an IHS under $T$.\end{proof}

\noindent Thus, we have the following immediate consequence.

\begin{proposition}\label{SScommutes1}Let $X$ be an infinite-dimensional complex Banach space and $T\in\mathcal{L}(X)$ satisfies the following conditions.
\begin{itemize}\item[(i)]  $\sigma_p(T^*)=\emptyset$; and
\item[(ii)]  $T$ commutes with a nonzero strictly singular (or compact) operator.\end{itemize}
Then $T$ admits an AIHS of defect $\leq 1$.\end{proposition}

\begin{proof}Let $S\in\mathcal{L}(X)$ be a nonzero strictly singular operator commuting with $T$.  Clearly, $S$ is not a multiple of the identity, and by the Spectral Theorem for strictly singular operators (cf., e.g., \cite[Theorem 7.11]{AA02}), $\sigma(S)$ and hence $\sigma_p(S)$ are countable subsets of $\mathbb{C}$.  Thus we can apply Lemma \ref{countable-commutes}.\end{proof}

Now, we are ready to prove Theorem \ref{SScommutes}.
\begin{proof}[Proof of Theorem \ref{SScommutes}]Consider a strictly singular operator $S\in\mathcal{L}(X)$ which commutes with $T$ and has infinite-dimensional range.  According to \cite[Proposition 3.5]{SW14}, there exists a finite-codimensional and $T$-invariant closed subspace $W$ of $X$ such that $\sigma_p(T|_W^*)=\emptyset$.  In fact, if we look at the proof to \cite[Proposition 3.5]{SW14}, we see that $W$ has the form $W=\overline{p(T)X}$ for some polynomial $p$.  Note that 
\[SW\subseteq\overline{SW}=\overline{S\overline{p(T)X}}=\overline{Sp(T)X}=\overline{p(T)SX}\subseteq\overline{p(T)X}=W.\]
so that $W$ is $S$-invariant.  Thus, for any $w\in W$ we have
\[S|_WT|_Ww=STw=TSw=T|_WS|_Ww\]
so that $S|_W$ commutes with $T|_W$.  Since $S$ is strictly singular, so is $S|_W$, and it is nonzero since $S$ has infinite-dimensional range.  Thus we can apply Proposition \ref{SScommutes1} to obtain an AIHS of defect $\leq 1$ for $T|_W$, and hence also for $T$.\end{proof}

\begin{remark}More generally, to obtain an AIHS under $T\in\mathcal{L}(X)$ of defect $\leq 1$, where $X$ is an infinite-dimensional complex Banach space, it is enough to find a commuting operator $S$ which has infinite-dimensional range, is not a multiple of the identity on any finite-codimensional subspace, and such that $\sigma_p(S)$ is countable.\end{remark}

\section{Structure of an operator without an AIHS}		
\label{structure}

In this section, we demonstrate a two-side shift-like structure for those hypothetical operators on an infinite-dimensional complex Banach space which fail to admit an AIHS of defect $\leq 1$.  
First, we shall describe a right-shift-like structure.  

\begin{theorem}Let $X$ be an infinite-dimensional complex Banach space, and let $T\in\mathcal{L}(X)$.  If $T$ fails to admit an AIHS of defect $\leq 1$ then there exists $e\in X$ such that $[T^ne]_{n=0}^\infty$ is finite-codimensional in $X$.\end{theorem}

\begin{proof}Suppose $T$ fails to have an AIHS of defect $\le 1$. By \cite[Remark 2.9]{MPR12}, we can find a finite-codimensional $T$-invariant subspace $W$ such that $\sigma(T|_W)$ is connected.  Let $U:=\lambda-T|_W\in\mathcal{L}(W)$ for some $\lambda\in\partial\sigma(T)$.  By \cite[Proposition 2.12]{SW14}, the null spaces $\mathcal{N}(U^n)$ are strictly increasing and finite-dimensional.  Thus we can apply \cite[Proposition 2.9]{SW14} to obtain $e\in W$ such that $V:=[U^ne]_{n=0}^\infty$ is infinite-dimensional.  Since $T$ fails to admit an AIHS of defect $\leq 1$, it must be that $V$ is finite-codimensional, which finishes the proof as
\[V=[(\lambda-T|_W)^ne]_{n=0}^\infty=[(T|_W)^ne]_{n=0}^\infty=[T^ne]_{n=0}^\infty.\]\end{proof}

To describe a left-shift-like structure in the Theorem \ref{4.13} we will need a couple of additional steps.

\begin{proposition}\label{surjectivity-spectra} Let $X$ be a complex Banach space and $T\in \mathcal{L}(X)$ be a continuous operator. Suppose that the null space $\mathcal{N}(T)$ is complemented in $X$ and let $W$ be a closed subspace such that $X=W\oplus\mathcal{N}(T)$.  Let $P_W:X\to W$ denote the continuous linear projection onto $W$ along $\mathcal{N}(T)$.  Then $P_WT|_W\in\mathcal{L}(W)$, and
\[\sigma_{su}(P_WT|_W)\subseteq\sigma_{su}(T)\subseteq\sigma_{su}(P_WT|_W)\cup\{0\}.\]
\noindent If furthermore $\sigma_{su}(T)$ is connected then $\sigma_{su}(T)=\sigma_{su}(P_WT|_W)$.\end{proposition}

\begin{proof}Write $S:=P_WT|_W$, and let $P_{\mathcal{N}(T)}:X\to\mathcal{N}(T)$ denote the continuous linear projection onto $\mathcal{N}(T)$ along $W$.

For the first inclusion, let $\lambda\in\rho_{su}(T)$ and $w\in W$.  Then there is $x_0\in X$ such that $(\lambda-T)x_0=w$.  Then
\begin{multline*}TP_Wx_0=T(P_Wx_0+P_{\mathcal{N}(T)}x_0)=Tx_0=\lambda x_0-w=\lambda(P_Wx_0+P_{\mathcal{N}(T)}x_0)-w\end{multline*}
\noindent and hence $SP_Wx_0=P_WTP_Wx_0=\lambda P_Wx_0-w$.  So $w=(\lambda-S)P_Wx_0$, which gives us $\lambda\in\rho_{su}(S)$.

For the second inclusion, let $\lambda\in\rho_{su}(S)\setminus\{0\}=(\sigma_{su}(S)\cup\{0\})^c$ and $x\in X$.  Then there is $w_0\in W$ such that $(\lambda-S)w_0=P_Wx$, and hence
\begin{multline*}x=P_Wx+P_{\mathcal{N}(T)}x=(\lambda-S)w_0+P_{\mathcal{N}(T)}x=(\lambda-T)w_0+P_{\mathcal{N}(T)}Tw_0+P_{\mathcal{N}(T)}x\\=(\lambda-T)w_0+\lambda P_{\mathcal{N}(T)}\lambda^{-1}(Tw_0+x)=(\lambda-T)(w_0+P_{\mathcal{N}(T)}\lambda^{-1}(Tw_0+x))\end{multline*}
\noindent so that $\lambda\in\rho_{su}(T)$.

To prove the last part of the Proposition, recall that the surjectivity spectrum is always a compact subset of $\mathbb{C}$ (cf., e.g., Theorem 2.42 in \cite{Ai04}).  Thus, if $\sigma_{su}(T)$ is connected, then we must have $\sigma_{su}(T)=\sigma_{su}(P_WT|_W)$.\end{proof}

As an aside, we might also ask whether the above Proposition can be improved to exclude the union with $\{0\}$.  The following example shows that this is not possible in general.

\begin{example}For any $1<p<\infty$, there exists a continuous linear operator $T\in\mathcal{L}(\ell_p)$ and a closed subspace $W$ such that $\ell_p=W\oplus\mathcal{N}(T)$ with $0\in\sigma_{su}(T)\cap\sigma_p(T)$ but $0\notin\sigma(P_WT|_W)$, where $P_W:\ell_p\to W$ denotes the continuous linear projection onto $W$ along $\mathcal{N}(T)$.\end{example}

\begin{proof}Let $(e_n)_{n=1}^\infty$ denote the canonical basis of $\ell_p$.  Set $R:\ell_p\to\ell_p$ as the unweighted right-shift operator, that is, the 1-bounded linear operator defined by $Re_n=e_{n+1}$ for all $n\in\mathbb{Z}^+$. 
Consider subspace $W:=[e_n]_{n=2}^\infty$ with basis projection $P_W$ and operator $T \in \mathcal{L}(W)$ defined by $T =(2-R)P_W$. 

Notice that operator $P_WT|_W=T|_W$, as an operator on $W\cong\ell_p$, acts exactly as $2-R$. Hence, the spectrum $\sigma(P_WT|_W)$ equals $\sigma(2-R)=\{\lambda:|\lambda-2|\le 1\}$ (cf., e.g., \cite[Example 6.21]{AA02}) which implies that zero is not in $\sigma(P_WT|_W)$. In particular, we have $\mathcal{N}(T)=[e_1]$ and, thus, $P_W$ is a projection onto $W$ along $\mathcal{N}(T)$. It is clear, that $T$ is not onto, which concludes the example.\end{proof}

Next, let us consider a sequence of null spaces $\mathcal{N}(T^n)$. These subspaces satisfy the inclusion $\mathcal{N}(T^n)\subseteq\mathcal{N}(T^{n+1})$ for every $n\in\mathbb{N}$. In addition, if, for some $n$, we have a decomposition $X=W_n \oplus \mathcal{N}(T^n)$
and $P_{W_n}$ is the respective projection onto $W_n$ along $\mathcal{N}(T^n)$, then 
$P_{W_n}\left[\mathcal{N}(T^m)\right]\subseteq\mathcal{N}(T^m)$ and hence
\begin{equation}\label{null-space-projection-identity}P_{W_n}\left[\mathcal{N}(T^m)\right]=\mathcal{N}(T^m)\cap W_n\;\;\;\text{ for any }m\geq n.\end{equation}
We can see this by applying $T^m$ to both sides of the decomposition $x=x_n+P_{W_n}x$ where $x$ is any vector from $\mathcal{N}(T^m)$ and $x_n$ is in $\mathcal{N}(T^{n})$. Let us use these observations to extend Proposition \ref{surjectivity-spectra}.

\begin{lemma}\label{4.7}Let $X$ be an infinite-dimensional complex Banach space and $T\in\mathcal{L}(X)$.  Then at least one of the following must be true.
\begin{itemize}\item[(i)]  $T$ admits an IHS.
\item[(ii)]  There exists a nonincreasing sequence
\[X=W_0\supseteq W_1\supseteq W_2\supseteq\cdots\]
\noindent of complements of $\mathcal{N}(T^n)$, and matching continuous linear projections $P_{W_n}:X\to W_n$ onto $W_n$ along $\mathcal{N}(T^n)$, such that for all $n\in\mathbb{N}$ we have
\[\sigma_{su}(P_{W_n}T|_{W_n})\subseteq\sigma_{su}(T)\subseteq\sigma_{su}(P_{W_n}T|_{W_n})\cup\{0\}.\]
\noindent  If furthermore $0\in \partial\sigma(T)$ and $\sigma(T)$ is connected, then, for all $n\in\mathbb{N}$,
\[\sigma_{su}(P_{W_n}T|_{W_n})=\sigma_{su}(T).\]\end{itemize}\end{lemma}

\begin{proof}Let us assume (i) is false.  Then, by \cite[Lemma 2.11]{SW14}, every $\mathcal{N}(T^n)$ is finite-dimensional and, hence, complemented in $X$. We set $W_0 = X$ and build the chain of $W_n$'s by induction. As $W_{n+1}$ we 
select any complement of the finite dimensional subspace $P_{W_n}\left[\mathcal{N}(T^{n+1})\right]= \mathcal{N}(T^{n+1}) \cap W_n$ in $W_n$.  Due to
\[\mathcal{N}(T^{n+1})\subseteq P_{W_n}\left[\mathcal{N}(T^{n+1})\right]\oplus\mathcal{N}(T^n)\subseteq\mathcal{N}(T^{n+1}),\]
$W_{n+1}$ is also a complement of $\mathcal{N}(T^{n+1})$ in $X$.  This, in particular, implies that the operator $P_{W_{n+1}}|_{W_n}$ is the projection from $W_n$ onto $W_{n+1}$ along $P_{W_n}\left[\mathcal{N}(T^{n+1})\right]$.

Next, we claim that $P_{W_n}\left[\mathcal{N}(T^{n+1})\right]=\mathcal{N}(P_{W_n}T|_{W_n})$. 
Since $T\left[\mathcal{N}(T^{n+1})\right] \subset\mathcal{N}(T^{n})$ is annihilated by $P_{W_n}$, and $P_{W_n}\left[\mathcal{N}(T^{n+1})\right]=\mathcal{N}(T^{n+1})\cap W_n$ by \eqref{null-space-projection-identity} above, we conclude that $P_{W_n}\left[\mathcal{N}(T^{n+1})\right]\subseteq\mathcal{N}(P_{W_n}T|_{W_n})$.

For the reverse inclusion, suppose $x\in\mathcal{N}(P_{W_n}T|_{W_n})$.  Then we can write
\[T^{n+1}x=T^n(P_{W_n}Tx+P_{\mathcal{N}(T^n)}Tx)=0+T^nP_{\mathcal{N}(T^n)}Tx=0,\]
and the claim is proved.
Thus, $P_{W_{n+1}}$ can be viewed as the projection from $W_n$ onto $W_{n+1}$ along $\mathcal{N}(P_{W_n}T|_{W_n})$.

Now, we are ready to prove the inclusion of spectra using Proposition \ref{surjectivity-spectra}. The case $n=0$ is 
clear and, then we have
\begin{multline*}\sigma_{su}(P_{W_{n+1}}T|_{W_{n+1}})=\sigma_{su}(P_{W_{n+1}}(P_{W_n}T|_{W_n})|_{W_{n+1}})\subseteq\sigma_{su}(P_{W_n}T|_{W_n})\\\subseteq\sigma_{su}(T)\subseteq\sigma_{su}(P_{W_n}T|_{W_n})\cup\{0\}\subseteq\sigma_{su}(P_{W_{n+1}}T|_{W_{n+1}})\cup\{0\}.\end{multline*}
\noindent This proves all but the last statement of (ii).

Now let us consider the case when $\sigma(T)$ is connected.  If equality did not hold, then we would have $\sigma_{su}(P_{W_n}T|_{W_n})=\sigma_{su}(T)\setminus\{0\}$ with $0\in\partial\sigma(T)\subset\sigma_{su}(T)$.  However, this is impossible since the surjectivity spectrum is a nonempty compact subset of $\mathbb{C}$ (cf., e.g., \cite[Theorem 2.42]{Ai04}).\end{proof}

Notice that in the proof of the last lemma, if $T$ has no IHS, we get
\begin{equation}\label{2}\mathcal{N}(P_{W_n}T|_{W_n})=\mathcal{N}(T^{n+1})\cap W_n.\end{equation}
Therefore, if $\mathcal{N}(T^n)=\mathcal{N}(T^{n+1})$ happens for any $n$, then $\mathcal{N}(P_{W_n}T|_{W_n})=\{0\}$.
This means $0$ is not an eigenvalue for operator $P_{W_n}T|_{W_n}$.  If, in addition, zero is a limit point of $\partial\sigma_{su}(T)$, then by Lemma \ref{4.7} together with compactness of the surjectivity spectrum (cf., e.g., \cite[Theorem 2.42]{Ai04}), we would get $0\in\partial\sigma_{su}(P_{W_n}T|_{W_n})$.  In that case, by \cite[Proposition 2.8]{SW14}, $P_{W_n}T|_{W_n}$ and hence $T$ would admit an AIHS of defect $\leq 1$.  This proves the following lemma.

\begin{lemma}\label{4.9}Let $X$ be an infinite-dimensional complex Banach space and $T\in\mathcal{L}(X)$ be such that $0\in\partial\sigma_{su}(T)$ is a limit point of $\partial\sigma_{su}(T)$, but $T$ fails to admit an AIHS of defect $\leq 1$.  Then $\dim\mathcal{N}(T^{n+1})/\mathcal{N}(T^n)>0$ for each $n\in\mathbb{N}$ and, consequently,
$\overline{\bigcup_{n=0}^\infty\mathcal{N}(T^n)}$ has infinite dimension and finite codimension.\end{lemma}

Now we are ready to prove this section's main result.

\begin{theorem}\label{4.13} Let $X$ be an infinite-dimensional complex Banach space and $T\in\mathcal{L}(X)$.  Then either $T$ admits an AIHS of defect $\leq 1$, or else there is a finite-codimensional closed subspace $V$, an eigenvalue $\lambda\in\sigma_p(T)\cap\sigma_p(P_VT|_V)$, and a linearly independent sequence $(v_n)_{n=1}^\infty\subseteq V$ such that $[v_n]_{n=1}^\infty=V$, and satisfying
\[Sv_{n+1}=v_n\;\;\text{ for all }n\in\mathbb{Z}^+,\;\;\text{ and }\;\;Sv_1=0,\]
where $P_V\in\mathcal{L}(X)$ is a projection onto $V$, and
\[S:=P_V(\lambda-T)|_V=\lambda-P_VT|_V\in\mathcal{L}(V).\]\end{theorem}
\begin{proof} Assume $T$ fails to admit an AIHS of defect $\leq 1$.  By Theorem \ref{countable} we may assume that $\sigma(T)$ and hence $\partial\sigma(T)$ are infinite.  Recall from \cite[Theorem 2.42]{Ai04} that the boundary of the spectrum is always contained in the boundary of the surjectivity spectrum, and that by \cite[Proposition 2.8]{SW14}, $\partial\sigma_{su}(T)\subseteq\sigma_p(T)$.  In particular, $\sigma_p(T)\cap \partial\sigma(T)$ should be infinite, and so, by shifting for convenience, we may assume that zero is a limit point of $\partial\sigma(T)$. Moreover, our target statement is now about operator $T$ instead of $\lambda - T$. 

Set $K_j:=\mathcal{N}(T^j)$ and observe that 
since $\widehat{T}:K_{j+2}/K_{j+1} \to K_{j+1}/K_j$ defined by $\widehat{T}(x + K_{j+1}):= Tx + K_j$ is a linear injective map
we have $\text{dim}(K_{j+2}/K_{j+1})\leq\text{dim}(K_{j+1}/K_j)$ for all $j\in\mathbb{Z}^+$.  
Moreover, by Lemma \ref{4.9}, $\lim_{j\to \infty}\text{dim}(K_{j+1}/K_j)$ must be a strictly positive number.
Hence, there exist integers $N$ and $r>0$ such that dim $(K_{N+n+1}/K_{N+n})=r$ for every $n>0$.

For convenience and without loss of generality, we may assume that $N=1$. 
Next, we use Lemma \ref{4.7} to select a sequence of spaces $(W_n)_{n=1}^{\infty}$ complementing $K_n$'s in $X$.  Recall from \eqref{2} that $\mathcal{N}(P_{W_n}T|_{W_n})=K_{n+1}\cap W_n$, so that $K_{n+1}\cap W_n$ must contain a nonzero vector, on pain of having $\sigma_p(P_{W_n}T|_{W_n})=\emptyset$.  However, the latter is impossible by Theorem \ref{countable}, since we have assumed $T$ and hence $P_{W_n}T|_{W_n}$ fails to admit an AIHS of defect $\leq 1$. 

We will select vectors $v_n \in \left(K_{n+1}\cap W_1\right)\backslash K_n$ inductively. We could guarantee the claim of the theorem if
for each $n$ we have $Tv_{n+1}=v_n + z_n$ with $z_n \in K_1$. Indeed, the last identity guarantees that
for $V=[v_n]_{n=1}^{\infty}$ we have $T(V)\subset V+K_1$. The fact that $v_n \in K_{n+1}\backslash K_n$
makes sure that $V$ is infinite-dimensional. Therefore, $V$ is finite-codimensional as $T$ does not have an IHS but $T(V+K_1)\subseteq V+K_1$ and dim $K_1 <\infty$. Finally, since $V \subset W_1$, there is a projection $P_V$ onto $V$ which annihilates $K_1$ and, hence, every $z_n$. 
That is, $P_VT v_{n+1}=v_n$ for every $n$ and $P_VTv_1 =0$. 

At last, let us select vectors $v_n$. 
Set $v_1$ to be any
vector in $(K_{2}\cap W_1)\backslash K_1 $. By the paragraph above, we will be done if we 
prove that for every $n>0$ and every $x\in \left(K_{n+1}\cap W_1\right)\backslash K_n$
there are vectors $v \in  \left(K_{n+2}\cap W_1\right)\backslash K_{n+1}$ and 
$z \in K_1$ such that $Tv = x+z$. To see this, consider linear map $\widetilde{T}:K_{n+2}\cap W_1 \to K_{n+1}/K_1$
naturally defined by $\widetilde{T}x :=Tx + K_1$. Notice that $\mathcal{N}(\widetilde{T})=K_2\cap W_1$ and, by \eqref{null-space-projection-identity}, $K_2 = K_1 \oplus (K_2 \cap W_1)$.  So, we conclude that $\text{dim }\mathcal{N}(\widetilde{T})=\text{dim }K_2/K_1 = r$, and hence, due to the fact that
$\text{dim }K_{n+1}/K_1=nr$, obtain
\[\text{dim range }\widetilde{T} =\text{dim }K_{n+2}\cap W_1-\text{dim }\mathcal{N}(\widetilde{T}) =(n+1)r-r= \text{dim }K_{n+1}/K_1,\]
which implies that $\widetilde{T}$ is a surjection. Therefore,
for each $x \in \left(K_{n+1}\cap W_1\right)\backslash K_n$ there is $v \in K_{n+2}\cap W_1$ such that 
$\widetilde{T}v = x + K_1$. The latter means $Tv = x+z$ for some $z \in K_1\subset K_n$. Notice that $v \notin K_{n+1}$ as, otherwise,
$Tv$ would have been in $K_n$ and same would have been true for $Tv -z = x$, whereas $x \notin K_n$.\end{proof}

\section{The essentially-invariant halfspace problem}

Does every operator acting on an infinite-dimensional Banach space admit an EIHS?  Let us call this the {\bf EIHS problem}.  In this section, we give an affirmative answer in the complex case.  As a corollary to the techniques we use, we will also observe that if $X$ is an infinite-dimensional complex Banach space then the set of all operators in $\mathcal{L}(X)$ which admit an AIHS of defect $\leq 1$ is dense in $\mathcal{L}(X)$.

To get the job done, we will need a straightforward fact, whose proof we include for completeness.

\begin{proposition}\label{compact-extension}Let $X$ be a Banach space and let $(y_n)_{n=1}^\infty$ be a seminormalized basic sequence in $X$ with basis constant $B\geq 1$.  Suppose $K:[y_n]_{n=1}^\infty\to X$ is a continuous linear operator satisfying $\sum_{n=1}^\infty\norm{Ky_n}<\infty$.  Then $K$ can be extended to a compact operator $\widetilde{K}\in\mathcal{K}(X)$ satisfying
\[\norm{\widetilde{K}}\leq 2B\sum_{n=1}^\infty\norm{Ky_n}.\]\end{proposition}

\begin{proof}Let $y^*_n\in[y_n]^*$ be coordinate functionals for the basis $(y_n)_{n=1}^{\infty}$.
Denote their norm-preserving extensions to the whole $X$ by the same $y_n^*$. Since $\|y_n^*\| \le 2B$ for each $n$ and
due to the absolute convergence of the series $\sum Ky_n$, we may define a continuous operator by the following:
\[\widetilde{K}x=\sum_{n=1}^\infty y_n^*(x)Ky_n\;\;\text{ for all }x\in X.\]
Since we are using coordinate functionals $\widetilde{K}$ is an extension of $K$ to $X$ and it is of the desired norm.
To show that $\widetilde{K}$ is compact, it is sufficient to observe that it is the norm-limit of its finite-rank partial sums.\end{proof}

\begin{theorem}\label{compact-restriction}Let $X$ be an infinite-dimensional complex Banach space, and let $T\in\mathcal{L}(X)$.  Then at least one of the following are true.
\begin{itemize}\item[(i)]  $T$ admits an AIHS of defect $\leq 1$; or
\item[(ii)]  There exists $\mu\in\mathbb{C}$ such that for any $\epsilon>0$ there exists a halfspace $Y\subseteq X$ and a compact operator $\widetilde{K}\in\mathcal{K}(X)$ with $\norm{\widetilde{K}}<\epsilon$, such that $(\mu-T)y=\widetilde{K}y$ for all $y\in Y$.\end{itemize}\end{theorem}

\begin{proof}Let us assume (i) is false.  In \cite[Theorem 3.5]{SW14}, it was shown that whenever an operator $T$ acting on a complex infinite-dimensional Banach space fails to admit an AIHS of defect $\leq 1$, there exists a finite-codimensional subspace $W$ for which $\sigma_p(T|_W^*)=\emptyset$.  Since the surjectivity spectrum is always a nonempty set (cf., e.g., \cite[Theorem 2.42]{Ai04}), we can find $\mu\in\sigma_{su}(T|_W)$, and define $S:=\mu-T|_W\in\mathcal{L}(W)$.  Notice that this means $\sigma_p(S^*)=\emptyset$ and $0\in\sigma_{su}(S)$.

We claim that for any finite-codimensional closed subspace $U\subseteq W$, the restriction $S|_U\in\mathcal{L}(U,W)$ cannot be bounded below.  For suppose otherwise, towards a contradiction, that is, suppose $S|_U$ is bounded below for some finite-codimensional closed subspace $U$.  Recall that bounded below operators have closed range (cf., e.g., \cite[Theorem 2.5]{AA02}).  In particular, $S|_U$ has closed range, and since $U$ is finite-codimensional, so does $S$.  Recall also that $\lambda\notin\sigma_p(S^*)$ just in case $\lambda-S$ has dense range (cf., e.g. \cite[Theorem 6.19]{AA02}).  Since $\sigma_p(S^*)=\emptyset$ this means $S$ must have dense range.  Along with its closed range, this gives us $SW=W$, which contradicts the fact that $0\in\sigma_{su}(S)$.  Thus the claim is proved, and $S|_U$ is not bounded below for any finite-codimensional closed subspace $U\subseteq W$. 

The last paragraph, by \cite[Proposition 2.c.4]{LT77}, implies that for every $\epsilon>0$ there exists an infinite dimensional subspace $V\subset W$ such that $S|_V$ is a compact operator and $\|S|_V\|< \epsilon$. 
Of course, we may assume that
 $V=[v_n]_{n=1}^{\infty}$ where $(v_n)_{n=1}^{\infty}$ is a normalized basic sequence. By taking appropriate subsequences,
 if needed, we may assume that $V$ is a half-space and that $Sv_n$ converges. It follows, that for seminormalized
 block basic sequence $y_n = v_{2n}-v_{2n-1}$ the space $Y=[y_n]$ is a half-space, operator $S|_Y$ is compact, and
 $\|S|_Y\|< \epsilon$. Moreover, we have $Sy_n \to 0$ and we may assume any rate of convergence. In particular, 
 we may apply Proposition \ref{compact-extension} to complete the proof.
\end{proof}

Our desired results for this section now follow immediately from the above.

\begin{corollary}Every operator acting on an infinite-dimensional complex Banach space admits an EIHS.\end{corollary}

\begin{corollary}Let $X$ be an infinite-dimensional complex Banach space.  Then the set
\[\left\{T\in\mathcal{L}(X):T\text{ admits an AIHS of defect}\leq 1\right\}\]
is norm-dense in $\mathcal{L}(X)$.\end{corollary}

\end{document}